\documentclass[a4paper,12pt]{amsart}
\usepackage{graphics}
\usepackage{graphicx}
\usepackage{amssymb}
\usepackage{amsmath}
\usepackage{amsthm}

\usepackage{multicol}

\usepackage[english]{babel}
\usepackage{cmap}

\usepackage{setspace}

\usepackage{enumerate}

\usepackage{setspace}
\makeatletter
 
  \@addtoreset{equation}{section}
\makeatother
\makeatletter
\def\tbcaption{\def\@captype{table}\caption}
\makeatother
\newtheorem{theorem}{Theorem}[section]
\newtheorem{lemma}[theorem]{Lemma}

\theoremstyle{definition}
\newtheorem{definition}[theorem]{Definition}

\newtheorem{proposition}[theorem]{Proposition}
\newtheorem{conjecture}{Conjecture}[section]

\theoremstyle{remark}
\newtheorem{remark}[theorem]{Remark}

\numberwithin{equation}{section}

\newcommand{\beq}[1]{\begin{equation}\label{#1}}
\newcommand{\eeq}{\end{equation}}

\def\KK{\mathbb K}
\def\QQ{\mathbb Q}
\def\ZZ{\mathbb Z}
\def\RR{\mathbb R}
\def\NN{\mathbb N}

\begin{document}
\title[On the $D(4)$-pairs $\{a, ka\}$ with $k\in \{2,3,6\}$]{\small On the $D(4)$-pairs $\{a, ka\}$ with $k\in \{2,3,6\}$}

\author[K. N. Ad\'edji, M. Bliznac Trebje\v{s}anin, A. Filipin and A. Togb\'e]{Kou\`essi Norbert Ad\'edji, Marija {Bliznac Trebje\v{s}anin}, Alan Filipin and Alain Togb\'e}
\date{\today}
\maketitle
\begin{abstract}
Let $a$ and $b=ka$ be positive integers with $k\in \{2, 3, 6\},$ such that $ab+4$
is a perfect square. In this paper, we study the extensibility of the $D(4)$-pairs $\{a, ka\}.$ More precisely, we prove that by considering three families of positive integers $c$ depending on $a,$ if $\{a, b, c, d\}$ is the set of positive integers which has the property that the product of any two of its elements increased by $4$ is a perfect square, then $d$ in given by $$d=a+b+c+\frac{1}{2}\left(abc\pm \sqrt{(ab+4)(ac+4)(bc+4)}\right).$$ As a corollary, we prove that any $D(4)$-quadruple which contains the pair $\{a, ka\}$ is regular.
\end{abstract}


\section{Introduction}\label{intr}

\begin{definition} Let $n\neq0$ be an integer. We call a set of $m$ distinct positive integers a $D(n)$-$m$-tuple or an $m$-tuple with the property $D(n)$, if the product of any two of its distinct elements increased by $n$ is a perfect square.
\end{definition}

For a $D(4)$-triple $\{a,b,c\}$, $a<b<c$, we define
$$d_{\pm}=d_{\pm}(a,b,c)=a+b+c+\frac{1}{2}\left(abc\pm \sqrt{(ab+4)(ac+4)(bc+4)}\right).$$
It is straightforward to check that $\{a,b,c,d_{+}\}$  is a $D(4)$-quadruple, which we will call a regular quadruple. A quadruple which is not regular is called an irregular quadruple. If $d_{-}\neq 0$ then $\{a,b,c,d_{-}\}$ is also a regular $D(4)$-quadruple with $d_{-}<c$. It is conjectured that an irregular quadruple doesn't exist.

\begin{conjecture}
Any $D(4)$-quadruple is regular.
\end{conjecture}

In this paper, we consider extensions of the $D(4)$-pairs $\{a,ka\}$, $k=2,3,6$ to a $D(4)$-quadruple $\{a,ka,c,d\}$ by following the method in \cite{ahpt}, where extensions of Diophantine pairs $\{a, ka\}$, with $k=3, 8$, were studied. 
We conjecture that $d=d_+(a,ka,c)$, i.e., that there is no irregular quadruple of this form. The validity of the conjecture is shown for some special pairs and triples. One of the results of interest for our case is the following lemma, which gives us a lower bound for the second element of the pair $b=ka$.
\begin{lemma}\cite[Lemma 2.2]{mbt}\label{lower_on_b}
Let $\{a,b,c,d\}$ be a $D(4)$-quadruple such that $a<b<c<d_+<d$. Then $b>10^5$.
\end{lemma}

If $\{a,ka\}$ is a $D(4)$-pair, then there exists $r\in \NN$ such that
\begin{equation}\label{pair eqn}
ka^2+4=r^2.
\end{equation}
Rewriting (\ref{pair eqn}) as a Pellian equation, yields 
\begin{equation}\label{pair Pellian} r^2-ka^2=4.
\end{equation}
The theory of Pellian equations guarantees that there is only one fundamental solution $(r_1,a_1)$ of (\ref{pair Pellian}), for any $k=2,3,6$, namely $(r_1,a_1)\in\{(6,4),(4,2),(10,4)\}$ (in that order). All solutions $(r_p,a_p)$ of the equation (\ref{pair Pellian}) are given by
\begin{align}\label{sequence-a_p}
\frac{r_p+a_p\sqrt{k}}{2}=\left(\frac{r_1+a_1\sqrt{k}}{2}\right)^p,\quad p\in\mathbb{N}.
\end{align}
It is easy to see that $\gcd(r,a)=2$ holds in every case. And since $b=ka > 10^5$ we can also deduce a lower bound for $a$. For $k=2$, we have $a\geq a_7=161564$, which also gives us $b=2a\geq 323128$. In the case $k=3$, we have $a\geq a_9= 81090$ and $b=3a\geq 3a_9\geq 243270.$ Finally, for $k=6$, the lower bounds $a\geq a_5=38804$, $b\geq 232824$ hold.

In general, if we extend a $D(4)$-pair $\{a,b\}$ to a $D(4)$-triple $\{a,b,c\}$ then there exist $s,t\in\mathbb{N}$ such that
\begin{align*}
    ac+4&=s^2,\\
    bc+4&=t^2.
\end{align*}
Combining these two equalities yields a Pellian equation
\begin{equation}\label{eqn: pell for c}
    at^2-bs^2=4(a-b).
\end{equation}
Its solutions $(t,s)$ are given by
$$
(t_{\nu}+s_{\nu}\sqrt{k})=(t_0+s_0\sqrt{k})\left(\frac{r+\sqrt{ab}}{2}\right)^{\nu},\quad \nu\geq 0,
$$
where $(t_0,s_0)$ is a fundamental solution of the equation (\ref{eqn: pell for c}) and $\nu$ is a nonnegative integer.

In \cite[Lemma 6.1]{mbt}, it has been shown that $(t_0,s_0)=(\pm2, 2)$ are the only fundamental solutions when $b\leq6.85a $, which is our case. We can represent the solutions $(t_{\nu},s_{\nu})$ as pair of binary recurrence sequences
\begin{align}
    &t_0=\pm 2,\ t_1=b\pm r,\ t_{\nu+2}=rt_{\nu +1}-t_{\nu},\label{sequence_t_nu}\\
    &s_0= 2,\ s_1=r\pm a,\ s_{\nu+2}=rs_{\nu +1}-s_{\nu},\ \nu\geq 0.\label{sequence_s_nu}
\end{align}
Since $c=\frac{s^2-4}{a}$, we give an explicit expression for the third element $c$ in the terms of $a$ and $b$ by
\begin{equation}\label{eqn:c_nu_general_exp}
    \resizebox{\textwidth}{!}{
    $c=c_{\nu}^{\pm}=\frac{4}{ab}\left\{\left(\frac{\sqrt{b}\pm\sqrt{a}}{2}\right)^2\left(\frac{r+\sqrt{ab}}{2}\right)^{2\nu}+\left(\frac{\sqrt{b}\mp\sqrt{a}}{2}\right)^2\left(\frac{r-\sqrt{ab}}{2}\right)^{2\nu}-\frac{a+b}{2}\right\},$
    }
\end{equation}
where $\nu\geq 0$ is an integer. From \cite[Proposition 1.8]{mbt}, if $a\geq 35$ then $c\geq c_{4}^-$ cannot hold, i.e., it remains to observe the cases $c\in\{c_1^{\pm},c_2^{\pm},c_3^{\pm}\}$. Let us list them in a  more suitable form
\begin{align*}
c_1^{\pm}&=a+b\pm 2r,\\
  c_2^{\pm}&=(ab+4)(a+b\pm2r)\mp4r,\\
  c_3^{\pm}&=(a^2b^2+6ab+9)(a+b\pm2r)\mp4r(ab+3).
  \end{align*}
Let us mention some observations in the case $c_1^-.$ Note that if $k\in \{2, 3\},$ then $c_1^-<a<b$ and in case $k=6$ we have $a<c_1^-<b.$ So, in these cases, we consider the $D(4)$-triple $\{a, b, c\}$ of the form $\{c_1^-, a, b\}$ or $\{a, c_1^-, b\}.$ 

The present paper deals with two closely related families, viz. those of $D(4)$-triples mentioned in this section. The outcome of our study is the theorem below, showing that each of the triples under scrutiny has a unique extension to quadruple. In particular, the next result shows that the above mentioned conjecture is true for the families of $c$ examined in this paper.
\begin{theorem}\label{main result}
Let $k$ and $\nu$ be positive integers such that $k\in \{2, 3, 6\}.$ If $\{a, b, c_{\nu}^\pm, d\}$ is a $D(4)$ quadruple with $b=ka$, then it is regular. In other words, we have $d=d_{\pm}.$
\end{theorem}

The organization of this paper is as follows. In Sections \ref{section_2} and \ref{section_3} of the paper, we will essentially prove the useful results to achieve our main goal. We devote Section \ref{section_4} to the proof of Theorem \ref{main result} and we will end in Section \ref{section_5} with some clarifications.

\section{Pellian equations and linear form in three logarithms}\label{section_2}

The goal of this section is to provide and prove the main technical tools used in our proof of Theorem~\ref{main result}. These tools are related to the search for the intersection of linear recurrent sequences and to linear form in logarithms.

\subsection{System of simultaneous Pellian equations}\label{subsec-2.1}

Let us observe an extension of a triple $\{a,b,c\}$ to a quadruple $\{a,b,c,d\}$
\begin{align*}
    ad+4&=x^2,\\
    bd+4&=y^2,\\
    cd+4&=z^2.
\end{align*}
By eliminating $d$ from these equations, we get a system of generalized Pellian equations
\begin{align}
az^2-cx^2&=4(a-c),\label{eqn:pell_ac}\\
bz^2-cy^2&=4(b-c),\label{eqn:pell_bc}\\
ay^2-bx^2&=4(a-b).\label{eqn:pell_ab}
\end{align}
Its solutions $(z,x)$, $(z,y)$ and $(y,x)$ satisfy
\begin{align}
    z\sqrt{a}+x\sqrt{c}&=(z_0\sqrt{a}+x_0\sqrt{c})\left(\frac{s+\sqrt{ac}}{2}\right)^m,\label{eqn:sol_pell_ac}\\
    z\sqrt{b}+y\sqrt{c}&=(z_1\sqrt{a}+y_1\sqrt{c})\left(\frac{t+\sqrt{bc}}{2}\right)^n,\label{eqn:sol_pell_bc}\\
    y\sqrt{a}+x\sqrt{b}&=(y_2\sqrt{a}+x_2\sqrt{b})\left(\frac{r+\sqrt{ab}}{2}\right)^l,\label{eqn:sol_pell_ab}
\end{align}
where $m,n,l$ are nonnegative integers and $(z_0,x_0)$, $(z_1,y_1)$ and $(y_2,x_2)$ are fundamental solutions of these equations.

Any solution to the system satisfies $z=v_m=w_n$, where ${v_m}$ and ${w_n}$ are recurrent sequences defined by
\begin{align*}
&v_0=z_0,\ v_1=\frac{1}{2}\left(sz_0+cx_0\right),\ v_{m+2}=sv_{m+1}-v_{m},\\
&w_0=z_1,\ w_1=\frac{1}{2}\left(tz_1+cy_1 \right),\ w_{n+2}=tw_{n+1}-w_n.
\end{align*}
The initial terms of these sequences are described in the next theorem.

\begin{theorem}\cite[Theorem 1.3]{mbt}\label{tm:inital_terms}
Suppose that $\{a,b,c,d\}$ is a $D(4)$-quadruple with $a<b<c<d$ and that  $w_m$ and  $v_n$  are defined as before. 
\begin{enumerate}
\item[i)] If  the equation $v_{2m}=w_{2n}$ has a solution, then $z_0=z_1$ and $|z_0|=2$ or $|z_0|=\frac{1}{2}(cr-st)$.
\item[ii)] If the equation $v_{2m+1}=w_{2n}$ has a solution, then $|z_0|=t$, $|z_1|=\frac{1}{2}(cr-st)$ and $z_0z_1<0$.
\item[iii)] If the equation $v_{2m}=w_{2n+1}$ has a solution, then $|z_1|=s$, $|z_0|=\frac{1}{2}(cr-st)$ and $z_0z_1<0$.
\item[iv)] If the equation $v_{2m+1}=w_{2n+1}$ has a solution, then $|z_0|=t$, $|z_1|=s$ and $z_0z_1>0$.
\end{enumerate}
Moreover, if $d>d_+$, case $ii)$ cannot occur.
\end{theorem}
The next result will also be useful in our case.
\begin{lemma}\cite[Lemma 14]{mbt}\label{lem:sequences_shift} 
Let $\{v_{z_0,m}\}$ denote a  sequence $\{v_m\}$ with an initial value $z_0$ and $\{w_{z_1,n}\}$ denote a sequence $\{w_n\}$ with an initial value $z_1$. It holds that $v_{\frac{1}{2}(cr-st),m}=v_{-t,m+1}$, $v_{-\frac{1}{2}(cr-st),m+1}=v_{t,m}$, for each $m\geq 0$ and $w_{\frac{1}{2}(cr-st),n}=w_{-s,n+1}$, $w_{-\frac{1}{2}(cr-st),n+1}=w_{s,n}$, for each $n \geq 0$.
\end{lemma}

These results prove the next lemma. 

\begin{lemma}\label{lem:inital_terms}
Assume that $\{a,b,c,c'\}$ is not a $D(4)$-quadruple for any $c'$ with $0<c'<c_{\nu-1}^{\pm}$. We have 
\begin{enumerate}[i)]
    \item If the equation $v_{2m}=w_{2n}$ has a solution, then $z_0=z_1=\pm 2$ and $x_0=y_1=2$.
    \item If the equation $v_{2m+1}=w_{2n+1}$ has a solution, then $z_0=\pm t$, $z_1=\pm s$, $x_0=y_1=r$ and $z_0z_1>0$.
\end{enumerate}
\end{lemma}

\begin{remark}\label{c_1-even}
If $c=c_1^{\pm}=a+b\pm 2r,$ then it is enough to observe the case $v_{2m}=w_{2n}$.
\end{remark}

Now, we observe the solutions of the system of equations (\ref{eqn:pell_bc}) and (\ref{eqn:sol_pell_ab}). More precisely, we will determine the intersections $y=A_n=B_l$ of sequences $(A_n)_n$ and $(B_l)_l$ defined by
\begin{align}
    &A_0=y_1,\ A_1=\frac{1}{2}(ty_1+bz_1),\ A_{n+2}=tA_{n+1}-A_n,\label{seq:An}\\
    &B_0=y_2,\ B_1=\frac{1}{2}(ry_2+bx_2),\ B_{l+2}=rB_{l+1}-B_l, \ n,j\geq 0.\label{seq:Bl}
\end{align}

The next lemma, which is a part of Lemma 2\@ in \cite{dujram}, gives us a description of the solutions of Pell equations.
\begin{lemma}\cite[Lemma 2]{dujram}\label{lem:pell_general}
If $(X,Y)$ is a positive integer solution to a generalized Pell equation
$$aY^2-bX^2=4(a-b),$$
with $ab+4=r^2$, we have
$$Y\sqrt{a}+X\sqrt{b}=(y_0\sqrt{a}+x_0\sqrt{b})\left(\frac{r+\sqrt{ab}}{2}\right)^n,$$
where $n\geq0$ is an integer and $(x_0,y_0)$ is integer solution of the equation such that
$$1\leq x_0 \leq \sqrt{\frac{a(b-a)}{r-2}}\quad \textit{and}\quad 1\leq |y_0| \leq \sqrt{\frac{(r-2)(b-a)}{a}}.$$
\end{lemma}

The next lemma is proved as \cite[Lemma 4]{ahpt}.

\begin{lemma}\label{lem:lem_4}
Assume that $\{a, b, c', c\}$ is not a $D(4)$-quadruple for any $c'$ with $0<c'<c^{\pm}_{\nu-1}$ and $b\geq 832824$. Then, $A_{2n}=B_{2l+1}$ has no solution. Moreover, if $A_{2n} = B_{2l}$
then $y_2 = 2$. In other cases, we have $y_2 =\pm 2.$
\end{lemma}
\begin{proof}
It is straightforward to check by induction that
\begin{align*}
    A_{2n}\equiv y_1\ (\bmod \ b),\quad &A_{2n+1}\equiv \frac{1}{2}(ty_1+bz_1)\ (\bmod \ b),\\
    B_{2l}\equiv y_2\ (\bmod \ b),\quad &B_{2l+1}\equiv \frac{1}{2}(ry_2+bx_2)\ (\bmod \ b).
\end{align*}
From Lemma \ref{lem:pell_general}, we have
$$|y_2|\leq \sqrt{\frac{(r-2)(b-a)}{a}}=\sqrt{\frac{(r-2)(k-1)a}{a}}=\sqrt{(k-1)(r-2)},$$
where we have used that $b=2k$, $k\in\{2,3,6\}$. This implies
$$|y_2|\leq \sqrt{(k-1)\sqrt{b^2/k+4}}\leq \begin{cases}0.85\sqrt{b},&\ k=2,\\
1.075\sqrt{b},&\ k=3,\\
1.43\sqrt{b},&\ k=6.\end{cases}$$
 \par

\textbf{Case 1: } If $A_{2n}=B_{2l}$ then $y_1\equiv y_2 \pmod  b$. From Lemma \ref{lem:inital_terms}, we have $y_1=2$, so $y_2\equiv 2\ (\bmod \ b)$. On the other hand, $y_2<1.43\sqrt{b}<0.5b$, for $b\geq 9$ so $y_2=2.$

\textbf{Case 2: } If $A_{2n}=B_{2l+1}$ we have $y_1=2$ and 
\begin{equation}\label{eqn:cong1}
2\equiv \frac{1}{2}(ry_2+bx_2)\ (\bmod \ b).
\end{equation}
Since $ab+4=r^2$, we know that $g=\gcd(b,r)\in\{1,2,4\}$. After multiplying congruence (\ref{eqn:cong1}) by $2$ and using the fact that $r^2\equiv 4 \ (\bmod \ b)$, we can divide the final congruence by $r$ and we get
\begin{equation}\label{eqn:cong2}
    r\equiv y_2\ \left(\bmod \ \frac{b}{g}\right).
\end{equation}
We can rewrite the upper bound on $|y_2|$ and get
$$|y_2|\leq  \begin{cases}0.002b,&\ k=2,\\
0.0022b,&\ k=3,\\
0.003b,&\ k=6,\end{cases}$$
where we have used a lower bound on $b$ for each value of $k$. Also, $r=\sqrt{ab+4}=\sqrt{b^2/k+4}=\sqrt{1/k+4/b^2}b$, so we can use lower bounds of $b$ to get those of $r$ in the terms of $b$. More precisely, $R_0 \cdot b<r<R_0+0.01\cdot b$, where 
$$\begin{cases} R_0=0.7,&\ k=2,\\
R_0=0.57,&\ k=3,\\
R_0=0.4,&\ k=6.\end{cases}$$
If $g=\gcd(b,r)=1$, for each $k$, we have $||y_2|-r|0.4b$, which is a contradiction with a lower bound for $|y_2|$. On the other hand, if $g=\gcd(b,r)\in\{2,4\}$ we have that $y_2=r+\frac{b}{4}\cdot p$, for some $p\in\mathbb{Z}$. If $p\geq 0$, we have $|y_2|\geq r$ and a contradiction as in the previous case. If $p\leq -4$, we have $|y_2|> 0.29b$. For the remaining cases, $p=-1,-2,-3$ we have $|y_2|> 0.13b,0.09b,0.04b$, which is a contradiction in each case. 

Other possibilities for the parities of the indices of the sequences $\{A_n\}$ and $\{B_l\}$ are solved similarly to \cite[Lemma 4]{ahpt}. So, we omit the details.
\end{proof}
Therefore, the fundamental solutions of equation \eqref{eqn:pell_ab} are $(y_2, x_2)=(\pm 2, 2).$ Finally, we need to look at $x=Q_m=P_l$, for some non-negative integers $m$ and $l,$ where the sequences $(Q_m)_{m\ge 0}$ and $(P_l)_{l\ge 0}$ are obtained using \eqref{eqn:sol_pell_ac} and \eqref{eqn:sol_pell_ab} and given by
\begin{align}
    &P_0=x_2,\ P_1=\frac{1}{2}\left(rx_2+ay_2 \right) ,\ P_{l+2}=rP_{l +1}-P_{l},\label{sequence_P_l}\\
    &Q_0= x_0,\ Q_1=\frac{1}{2}\left(sx_0+az_0 \right),\ Q_{m+2}=sQ_{m +1}-Q_{m}.\label{sequence_Q_m}
\end{align}
From the above, we conclude that for the equation $x=P_l=Qm,$ the following
two possibilities exist:\\
\textbf{Type $a:$} If $l\equiv m\equiv 0\pmod{2},$ then $z_0=\pm 2,$ $x_0=y_2=2$ and $x_2=2.$\\
\textbf{Type $b:$} If $m\equiv 1\pmod{2},$ then $z_0=\pm t,$ $x_0=r,$ $y_2=\pm 2$ and $x_2=2.$

For the rest of this study, we will carefully examine the following equation 
\begin{equation}\label{eq-Q_m=P_l}
x=Q_m=P_l,
\end{equation}
while using the fundamental solutions of Types $a$ and $b.$ As we mentioned in Remark \ref{c_1-even}, we only need to consider solutions in Type $a$ if $c=c_1^\pm$ since $\dfrac{1}{2}(cr-st)=\dfrac{1}{2}\left((a+b\pm 2r)r-(r\pm a)(b+\pm r) \right)=\pm 2.$

\subsection{A linear form in three logarithms}\label{subsec-2.2}
Solving recurrences \eqref{sequence_P_l} and \eqref{sequence_Q_m}, we obtain
\begin{align}
\notag P_l&=\frac{1}{2\sqrt{b}}\left((y_2\sqrt{a}+x_2\sqrt{b})\alpha^l-(y_2\sqrt{a}-x_2\sqrt{b})\alpha^{-l} \right),\\
\notag Q_m&=\frac{1}{2\sqrt{c}}\left((z_0\sqrt{a}+x_0\sqrt{c})\beta^m-(z_0\sqrt{a}-x_0\sqrt{c})\beta^{-m} \right),
\end{align} 
where 
\begin{equation}\label{alpha-beta}
\alpha=\dfrac{r +\sqrt{ab}}{2} \quad \mbox{and}\quad \beta=\dfrac{s+\sqrt{ac}}{2}.
\end{equation}
Let us define 
\begin{equation}
\gamma=\frac{\sqrt{c}(y_2\sqrt{a} + x_2\sqrt{b})}{\sqrt{b}(z_0\sqrt{a}+x_0\sqrt{c})}.
\end{equation}
We follow the strategy used in \cite{ahpt} with some improvements and define the following linear form in three logarithms 
\begin{equation}\label{Lambda}
\Lambda = l\log \alpha - m\log \beta + \log \gamma.
\end{equation}

\begin{lemma}\label{bound-Lambda}
\begin{enumerate}[1)]
  \item If the equation $P_l=Q_m$ has a solution $(l,m)$ of Type a with $m\ge 1$, then
$$
0<\Lambda < 11.7 \beta^{-2m}.
$$  
   \item If the equation $P_l=Q_m$ has a solution $(l,m)$ of Type b with $m\ge 1$, then
$$
0<\Lambda < 4.4a^2 \beta^{-2m}.
$$
\end{enumerate}
\end{lemma}
\begin{proof}
We define 
$$
E=\frac{y_2\sqrt{a}+x_2\sqrt{b}}{\sqrt{b}}\alpha^{l}\quad \mbox{and} \quad
F=\frac{z_0\sqrt{a}+x_0\sqrt{c}}{\sqrt{c}}\beta^{m}.
$$
Easily $E,\,F>1$ if $l, m\ge 1$. Then, the equation $P_l=Q_m$ becomes
\begin{align}\label{ZU3}
E + 4\left(\frac{b-a}{b}\right)E^{-1} = F + 4\left(\frac{c-a}{c}\right)F^{-1}.
\end{align}
Because $c>b> 10^5$, we have $\frac{c-a}{c}>\frac{b-a}{b}$.
It follows that
\begin{equation}\label{eq:P-Q}
E + 4\left(\frac{b-a}{b}\right)E^{-1}>F + 4\left(\frac{b-a}{b}\right)F^{-1}
\end{equation}
and hence
$$
(E-F)\left(EF-4\left(\frac{b-a}{b}\right)\right) >0.
$$
Therefore, $E>F$. Moreover, by \eqref{ZU3}
we have
$$
0<E-F<4\left(\frac{c-a}{c}\right)F^{-1}<4F^{-1}.
$$
It follows that $\Lambda>0$, with
$$
\Lambda = \log
\frac{E}{F}=\log\left(1+\frac{E-F}{F}\right)<\frac{E-F}{F}<4F^{-2}.
$$

\textbf{Type $a$:}
\begin{align*}
\Lambda &<4\frac{c}{(\pm 2\sqrt{a}+2\sqrt{c})^2}\beta^{-2m}=\frac{c}{(\pm \sqrt{a}+\sqrt{c})^2}\beta^{-2m}\\
&<\dfrac{k}{(\sqrt{k}-1)^2} \beta^{-2m}<11.7\beta^{-2m},\; \text{for}\; c>b=ka> 10^5.
\end{align*}
Note that the above bound also holds for the $D(4)$-triple of the form $\{c_1^-, a, b\}$ and $\{a, c_1^-, b\}.$ 

\textbf{Type $b$:}\\
$\bullet$ The case $z_0=t.$ We have
$$
\Lambda <4\frac{c}{( t\sqrt{a}+r\sqrt{c})^2}\beta^{-2m}<\frac{4}{r^2}\frac{c}{( \sqrt{a}+\sqrt{c})^2}\beta^{-2m}<\beta^{-2m},
$$
for $c>b> 10^5$ and $r>100$.\\
$\bullet$ The case $z_0=-t.$ From \eqref{ZU3}, we get 
\begin{align*}
F=E + 4\left(\frac{b-a}{b}\right)E^{-1} - 4\left(\frac{c-a}{c}\right)F^{-1}
&>E-4\left(\frac{c-a}{c}\right)F^{-1}\\
&>E-4\left(\dfrac{c-a}{c}\right)>0.
\end{align*}
The above inequality comes from the fact that $F>1.$ Thus, we obtain
$$
F^{-1}<\left( E-4\left( \dfrac{c-a}{c}\right)\right)^{-1}.
$$
Therefore, we have
\begin{align}\label{Zu3}
\nonumber E-F&=4\left(\frac{c-a}{c}\right)F^{-1} - 4\left(\frac{b-a}{b}\right)E^{-1}\\
   &<4\left(\dfrac{c-a}{c} \right)\left( E-4\left( \dfrac{c-a}{c}\right)\right)^{-1}-4\left(\dfrac{b-a}{b}\right)E^{-1}.
\end{align}
Moreover, in Type b and for $m\ge 3,$ we have 
\begin{align}\label{Zu4}
\nonumber F&\ge \dfrac{r\sqrt{c}-t\sqrt{a}}{\sqrt{c}}\beta^3=\dfrac{4(c-a)}{\sqrt{c}(r\sqrt{c}+t\sqrt{a})}\cdot\left( \dfrac{s+\sqrt{ac}}{2}\right)^3\\
&>\dfrac{4(c-a)}{\sqrt{c}\cdot 2r\sqrt{c}}\cdot (\sqrt{ac})^3>4(c-a),
\end{align}
which implies $E>F>4(c-a)$ and then 
\begin{align}\label{Zu5}
\dfrac{c-a}{c}\left( E-4\left( \dfrac{c-a}{c}\right) \right)^{-1}<E^{-1}.
\end{align}
When $m=1,$ we can easily see that  $P_l\neq Q_1$ by following what is done in the proof of \cite[Lemma 3.3]{aft}. So, combining \eqref{Zu3} and \eqref{Zu5}, we obtain
$$
E-F<4E^{-1}-4\left(\dfrac{b-a}{b} \right)E^{-1}=\dfrac{4a}{b}E^{-1}<\dfrac{4}{k}F^{-1}.
$$
Therefore, one can see that
$$
\Lambda=\log \dfrac{E}{F}=\log\left(1+\dfrac{E-F}{F} \right)<\dfrac{E-F}{F}<\dfrac{4}{k}F^{-2}
$$
and 
\begin{align*}
\dfrac{4}{k}F^{-2}=\dfrac{4}{k}\cdot \dfrac{c}{(r\sqrt{c}-t\sqrt{a})^2}\beta^{-2m}<\dfrac{c^2r^2}{k(c-a)^2}\beta^{-2m}.
\end{align*}
Using the fact that $c>b=ka>10^5,$ we get 
$$
r^2=ka^2+4<1.1ka^2\quad \text{and}\quad \dfrac{c^2}{k(c-a)^2}<\dfrac{k}{(k-1)^2}.
$$
Hence, $\Lambda=\log\dfrac{E}{F}<4.4a^2\beta^{-2m}.$ Considering all cases in Type $b,$ we have $\Lambda<4.4a^2\beta^{-2m}$. This completes the proof of Lemma~\ref{bound-Lambda}.
\end{proof}

We now have the next result whose proof is similar to that of part 2) in \cite[Lemma 8]{ahpt}.
\begin{lemma}
If the equation $x=Q_m=P_l$ has a solution $(l,m)$ with $m\ge 3,$ then $m\le l.$
\end{lemma}

For any nonzero algebraic number $\alpha$ of degree $d$ over $\QQ$, whose minimal polynomial over $\ZZ$ is $a_0 \prod_{j=1}^{d}(X-\alpha^{(j)}),$ we denote by 
$$ 
 h(\alpha)= \frac{1}{d}\left(\log |a_0|+\sum_{j=1}^{d}\log\max\left(1,\left|\alpha^{(j)} \right| \right) \right)
$$ 
its absolute logarithmic height. We recall the following result due to Matveev \cite{Matveev:2000}.
\begin{lemma}\label{Matveev}
Denote by $\alpha_1,\ldots,\alpha_j$ algebraic numbers, not $0$ or $1,$ by $\log\alpha_1,\ldots,$ $\log\alpha_j$ determinations of their logarithms, by $D$ the degree over $\QQ$ of the number field $\KK=\QQ(\alpha_1,\ldots,\alpha_j),$ and by $b_1,\ldots,b_j$ integers. Define $B=\max\{|b_1|,\ldots,|b_j|\},$ and $A_i=\max\{Dh(\alpha_i),|\log\alpha_i|,0.16\}$ ($1\leq i\leq j$), where $h(\alpha)$ denotes the absolute logarithmic Weil height of $\alpha.$ Assume that the number 
$$
\Lambda= b_1\log\alpha_1+\cdots+b_n\log\alpha_j 
$$ 
does not vanish; then 
$$ 
|\Lambda|\geq \exp\{-C(j,\chi)D^2A_1\cdots A_j\log(eD)\log(eB)\},
$$ 
where $\chi=1$ if $\KK \subset\RR$ and $\chi=2$ otherwise and 
$$ 
C(j,\chi)=\min\left\{\frac{1}{\chi}\left( \frac{1}{2}ej\right)^{\chi}30^{j+3}j^{3.5},2^{6j+20}\right\}.
$$
\end{lemma}

\begin{proposition}\label{proposition-1}
Assume that $c\in \{c_1^+, c_2^\pm\}.$ If $Q_m=P_l,$ then
\begin{align*}
\frac{l}{\log(el)}&< 3.36\cdot 10^{13}\cdot\log^2(10.7c^2),\quad \text{with solutions of Type a},\; \\
\frac{l}{\log(el)}&< 6.44\cdot 10^{13}\cdot\log^2(10.7c^2), \quad \text{with solutions of Type b}.
\end{align*}
If $Q_m=P_l,$ with $c=c_1^-,$ then we get 
\begin{align*}
\frac{l}{\log(el)}< 4.5\cdot 10^{13}\cdot\log^2(72a),\quad \text{if}\; k=2,3 
\end{align*}
and 
\begin{align*}
\frac{l}{\log(el)}< 9\cdot 10^{13}\cdot\log^2(27a),\quad \text{if}\; k=6.
\end{align*}
\end{proposition}
\begin{proof}
We apply Lemma \ref{Matveev} with $j=3$ and $\chi=1$ to the form \eqref{Lambda}. Here, we take 
$$
 D=4, b_1=l, b_2=-m, b_3=1, \alpha_1=\alpha, \alpha_2=\beta\; \text{and}\; \alpha_3=\gamma.
$$ 
Since $m\le l,$ we can take $B=l.$ Also, we have
$$ 
h(\alpha_1)=\frac{1}{2}\log\alpha\quad \text{and} \quad h(\alpha_2)=\frac{1}{2}\log\beta.
$$ 
$\bullet$ In the case $c\in \{c_1^+, c_2^\pm\}$ with $b=ka,$ we have $c-a>\left(1-\dfrac{1}{k} \right)c$ and $r<\dfrac{1}{3}c.$ Moreover, the conjugates of $\alpha_3$ are
$$
\frac{\sqrt{c}(y_2\sqrt{a} \pm x_2\sqrt{b})}{\sqrt{b}(z_0\sqrt{a}\pm x_0\sqrt{c})},
$$
and the leading coefficient of the minimal polynomial of $\alpha_3$ is $a_0=16k^2(c-a)^2.$
We proceed with the following estimates

\begin{align*}
 h(\alpha_3)& \leq \frac{1}{4} \left[\log(16k^2(c-a)^2)+4\log\frac{\max\{|\sqrt{c}(y_2\sqrt{a}\pm x_2\sqrt{b})|\}}{\min\{|\sqrt{b}(z_0\sqrt{a}\pm x_0\sqrt{c})|\}} \right]\\
 &= \frac{1}{4}\left[\log(16k^2(c-a)^2)+ 4\log\frac{2\sqrt{c}(1+\sqrt{k})}{\sqrt{k}(-t\sqrt{a}+r\sqrt{c})} \right]\\
 &<  \frac{1}{4}\log\left[\frac{2^4r^4c^4(1+\sqrt{k})^4}{(c-a)^2}\right] < \frac{1}{4}\log\left(\frac{k^2(1+\sqrt{k})^2\cdot c^6}{(\sqrt{k}-1)^2} \right) \\ &< \frac{1}{4}\log(1223c^6)<\dfrac{3}{4}\log (10.7c^2).
\end{align*}
Thus, we can take 
$$
A_1=2\log\alpha,\quad A_2=2\log\beta,\quad A_3=3\log(10.7c^2).
$$ 
Applying Lemma \ref{Matveev}, we get 
\begin{align}\label{Matveev_0}
\log|\Lambda|>-1.3901\cdot 10^{11}\cdot16\cdot12 \cdot\log\alpha\cdot\log\beta\cdot\log(10.7c^2)\cdot\log(4e)\cdot\log(el).
\end{align}

From each inequality of Lemma~\ref{bound-Lambda}, we have $l\log\alpha< 2m\log\beta$ and also
 \begin{align}\label{log-lambda-upp1}
 \log|\Lambda|<-1.9m\log\beta \quad \text{in Type a},\; \text{with}\; m\ge 1
\end{align} 
and
\begin{align}\label{log-lambda-upp2}
 \log|\Lambda|<-0.99m\log\beta \quad \text{in Type b},\; \text{with}\; m\ge 2.
\end{align}
Moreover,
\begin{align}\label{log-beta-upp}
\log\beta <\frac{1}{2}\log(10.7c^2).
\end{align}
Combining \eqref{Matveev_0}, \eqref{log-lambda-upp1}, \eqref{log-lambda-upp2}, and  \eqref{log-beta-upp}, we respectively get according to Type a and Type b the following inequalities
\begin{align*}
\frac{l}{\log(el)}< 3.36\cdot 10^{13}\cdot\log^2(10.7c^2)\; \text{and}\; \frac{l}{\log(el)}< 6.44\cdot 10^{13}\cdot\log^2(10.7c^2).
\end{align*}

$\bullet$ In the case $c=c_1^-$ with solutions of Type a, we similarly get the following inequalities
\begin{align*}
\frac{l}{\log(el)}< 4.5\cdot 10^{13}\cdot\log^2(72a),\quad \text{if}\; k=2,3 
\end{align*}
and 
\begin{align*}
\frac{l}{\log(el)}< 9\cdot 10^{13}\cdot\log^2(27a),\quad \text{if}\; k=6.
\end{align*}
This ends the proof.
\end{proof}

\section{Lower bounds of $m$ and $l$ in terms of $a$}\label{section_3}

In this section we will apply congruence relations to obtain some lower
bounds for the indices $m$ and $l$ satisfying the equation $x=Q_m=P_l.$

\begin{lemma}\label{lem_6} If $a$ is odd, then 
\begin{align}
 Q_{2m}&\equiv  x_0+\dfrac{1}{2} a(cx_0m^2+sz_0m)\pmod{a^2},\\
P_{2l}&\equiv  x_2+\dfrac{1}{2} a(bx_2l^2+ry_2l)\pmod{a^2} .
\end{align}
If $a$ is even, then 
\begin{align}
 Q_{2m}&\equiv  x_0+\dfrac{1}{2} a(cx_0m^2+sz_0m)\pmod{\dfrac{1}{2} a^2},\\
P_{2l}&\equiv  x_2+\dfrac{1}{2} a(bx_2l^2+ry_2l)\pmod{\dfrac{1}{2} a^2} .
\end{align}
\end{lemma}
\begin{proof}
The proof is similar to that of \cite[Lemma 15]{ahpt}.
\end{proof}
We now consider the following result.

\begin{lemma}\label{lower-Type-a}
If the equation $P_l=Q_m$ has a solution $(l,m)$ of Type~$a$, then we have  
$$
l\geq \frac{1}{12}\left(-2+\sqrt{4+3\sqrt{a}}\right).
$$ 
\end{lemma}
\begin{proof} Let $\lambda \in \{1, 2, 3\}$ and $\nu=1, 2 ,3,$ we have $c=c_\nu^\pm \equiv \pm 2\lambda r\pmod{a}.$ From  \eqref{sequence_s_nu}, we see that $s\equiv 2,  r\pmod{a}.$ We also get that $b=ak\equiv 0\pmod{a}.$
Using Lemma~\ref{lem_6}, we have in all cases with solutions in Type a
$$  
cx_0m^2+ sz_0m\equiv bx_2l^2+ry_2l\pmod{a}.
$$
Hence, we get 
$$  
\pm 4\lambda r m^2\pm 4m\equiv 2rl\pmod{a},\quad \mbox{if}\quad s\equiv 2\pmod{a}
$$
and 
$$  
\pm 4\lambda m^2\pm 2m-2l\equiv 0\pmod{\dfrac{a}{\gcd (a,r)}},\quad \mbox{if}\quad s\equiv r\pmod{a}.
$$

\textbf{The case $s\equiv r\pmod{a}.$} Recall that in our case $\gcd (a,r)=2.$ Thus,
$$
\left| \pm 4\lambda m^2\pm 2m-2l\right|\ge \dfrac{a}{\gcd (a,r)}=\dfrac{a}{2}.
$$
Note also that $m\le l$ and $\lambda\le 3.$ So, we get $12l^2+4l\ge \dfrac{a}{2},$ which implies 
\begin{align}\label{1}
l\ge \dfrac{1}{12}\left(-2+\sqrt{4+6a} \right).
\end{align}

\textbf{The case $s\equiv 2\pmod{a}.$} In this case, by multiplying the congruence obtained by $r$ and using the fact that $r^2\equiv 4\pmod{a},$ we get  
\begin{align}\label{eqq_38}
 \pm 16\lambda m^2 \pm 4mr-8l\equiv 0\pmod{a}.
\end{align}
Because $r^2\equiv 4\pmod{a},$ we conclude that $r\equiv \pm2 \pmod{a'}$ for some $a'$
which is a divisor of $a$ and $a'\ge \sqrt{a}.$ It follows that 
\begin{align*}
 \pm 16\lambda m^2 \pm 8m-8l\equiv 0\pmod{a'}.
\end{align*}
Therefore, we deduce that
$$
\left| \pm 16\lambda m^2 \pm 8m-8l\right|\ge a'\ge \sqrt{a}.
$$
Using again $m\le l$ and $\lambda\le 3,$ we get 
$$
48l^2+16l\ge \sqrt{a},
$$
which implies 
\begin{align}\label{2}
l\ge \dfrac{1}{12}\left(-2+\sqrt{4+3\sqrt{a}} \right).
\end{align}
Combining the inequalities \eqref{1} and \eqref{2}, we obtain the desired inequality. This completes the proof.
\end{proof}

Let us define \begin{equation}
\alpha^\nu=\left(\dfrac{r+\sqrt{ab}}{2} \right)^\nu=T_\nu + U_\nu\sqrt{ab},
\end{equation}
where $(T_\nu, U_\nu)$ is the $\nu$-th positive rational solution to the Pell equation
$$
T^2-abU^2=T^2-(r^2-4)U^2=1.
$$
It is easy to show by induction that
\begin{align}
\label{sequence-T}T_0&=1,\; T_1=\dfrac{r}{2},\; T_{\nu+2}=rT_{\nu+1}-T_\nu,\\
\label{sequence-U}U_0&=0,\; U_1=\dfrac{1}{2},\; U_{\nu+2}=rU_{\nu+1}-U_\nu,
\end{align}
for $ \nu \geq0.$ From \eqref{sequence_t_nu}, \eqref{sequence_s_nu}, \eqref{sequence-T}, and \eqref{sequence-U}, we have 
\begin{equation}
s=s_\nu^\pm=2T_\nu\pm 2aU_\nu \mbox{ and } t=t_\nu^\pm=\pm 2T_\nu+ 2bU_\nu.
\end{equation}
Considering congruences modulo $r,$ we get
$$
 2T_\nu \equiv \pm 2 \pmod{r},\; 2U_\nu \equiv 0 \pmod{r} \quad \mbox{if}\quad \nu\equiv 0\pmod{2},
 $$
 and 
 $$
 2T_\nu \equiv  0 \pmod{r},\; 2U_\nu \equiv \pm 1 \pmod{r} \quad \mbox{if}\quad \nu \equiv 1\pmod{2},
$$
which implies that $s\equiv \pm 2,\pm a\pmod{r}.$ The case $s\equiv \pm a\pmod{r}$ leads to a contradiction if $c=c_2^\pm.$ Let us prove it. In this case, we have 
$$
s^2=ac_2^\pm +4\equiv  a^2\pmod{r}.
$$
Multiplying the above congruence by $k$ and using the fact that $ka^2=r^2-4 \equiv -4\pmod{r},$ we get 
\begin{align}\label{cong-1}
kac_2^\pm\equiv  -4k-4\pmod{r}.
\end{align}
Furthermore, from $c=c_2^\pm$ defined in this paper, we have 
\begin{align}\label{cong-2}
kac_2^\pm\equiv 0\pmod{r}.
\end{align}
So, combining \eqref{cong-1} and \eqref{cong-2}, we get $4k+4\equiv  0\pmod{r}$ and therefore 
$$
r\mid 4k+4\in \{12, 16, 28\},
$$
which is not possible in our case since $r>28$, i.e., $b>10^5.$  In conclusion, we summarize what we have proved by the following result.
\begin{proposition}
Let $c=c_2^{\pm}=(ab+4)(a+b\pm2r)\mp4r.$ There is no $D(4)$-triple $\{a, b, c\}$  if $s\equiv \pm a\pmod{r}.$
\end{proposition}

Therefore, for solutions of Type $b$ with $s\equiv \pm 2\pmod{r},$ we get the following key result.
\begin{lemma}\label{lower-Type-b}
Assume that $c=c_2^\pm.$ If the equation $P_l=Q_m$ has a solution $(l,m)$ of Type~$b$, then we have
$$
m>  \begin{cases}
(3\sqrt{2}-4)a/4,&\ k=2,\\
(2\sqrt{3}-3)a/2,&\ k=3,\\
(5\sqrt{6}-12)a/4,&\ k=6.
\end{cases}
$$
\end{lemma}
\begin{proof}
By induction, we obtain 
$$
P_l\equiv \pm 2, \pm a\pmod{r}\quad\text{and}\quad Q_m \equiv \dfrac{1}{2}am z_0\pmod{r}.
$$ 
Because, $t=t_\nu^\pm=\pm 2T_\nu+ 2bU_\nu \equiv \pm 2,\pm b\pmod{r}$ and $z_0=\pm t,$ we would get 
$$
Q_m \equiv \pm am, \pm 2m\pmod{r}.
$$

$\bullet$ \textbf{The case $Q_m \equiv \pm 2m\pmod{r}.$ } Thus, first the equation $Q_m=P_l$ implies 
$$
\pm 2m\equiv  \pm 2\pmod{r},
$$
which becomes $\pm m\pm 1\equiv 0\pmod{r/2}$ and then $m\ge r/2-1$. Secondarily, we see that
$$
\pm 2m\equiv  \pm a\pmod{r},
$$
which leads to $2m+a\ge r.$ It follows that 
$$
m\ge \dfrac{1}{2}(r-a)>\dfrac{1}{2}(\sqrt{k}-1)a,\quad \text{for}\; k\in \{2, 3, 6\}.
$$

$\bullet$ \textbf{The case $Q_m \equiv \pm am\pmod{r}.$} Combining this with $P_l\equiv  \pm a\pmod{r},$ the equation $Q_m=P_l$ implies $am\equiv \pm a\pmod{r},$ which, using $\gcd (a,r)=2$, gives $m\equiv \pm 1\pmod{r}$ and $m\ge r/2-1.$  Now, we use the fact that $ P_l\equiv \pm 2 \pmod{r}$ to see that the equation $Q_m=P_l$ implies 
\begin{align}
ma \equiv \pm 2 \pmod{r}.
\end{align}
Multiplying the above congruence by $ka$ and adding $4m,$ we would get
$$
m(ka^2+4) \equiv 4m\pm 2ka \pmod{r},
$$
which gives
\begin{align}
2ka\pm 4m \equiv 0\pmod{r}.
\end{align}
Then, $m\ge \dfrac{1}{2}ka$ or the left hand side is positive and we get for $k=2$ the possibilities:
\begin{align}
\label{m-lower-1} 4a+4m&=3r, 4r, 5r, \ldots\\
\label{m-lower-2}4a-4m&=r, 2r.
\end{align}
From \eqref{m-lower-1}, we have $4a+4m\ge 3r$ and using $r=\sqrt{2a^2+4}>a\sqrt{2}$ we get 
\begin{align}\label{lower-k=2}
m>\dfrac{1}{4}(3\sqrt{2}-4)a.
\end{align}
In the case \eqref{m-lower-2}, we have $4a-4m\le 2r.$ Since $r<a\sqrt{2.1},$ we obtain 
\begin{align}
m>\dfrac{1}{2}(2-\sqrt{2.1})a.
\end{align}
We conclude that for $k=2,$ the relation \eqref{lower-k=2} holds in all cases. The other cases $(k=3,6)$ can be treated in the same way. So, we omit them.
\end{proof}

\section{Proof of Theorem \ref{main result}}\label{section_4}

In this section, we completely give the proof of Theorem \ref{main result} in two subsections according to the values of $c.$ 

\subsection{Proof of Theorem \ref{main result} with $c=c_1^\pm, c_2^\pm$} \label{subsec4.1}
We start by what follows.

\textbf{$\bullet$ The case $b=2a.$} From \eqref{sequence-a_p}, we easily get the following relation
\begin{align}\label{a_p_1}
a=a_p=\dfrac{1}{\sqrt{2}}\left((3+2\sqrt{2})^p-(3-2\sqrt{2})^p \right),
\end{align}
which gives using Proposition~\ref{proposition-1}, Lemmas~\ref{lower-Type-a} and \ref{lower-Type-b} the following result.

\begin{lemma}
\begin{enumerate}[1)]
\item For a $D(4)$-triple $\{a,2a,c_1^\pm\}$ with $a=a_p$ $(p\ge 1),$ if the equation $P_l=Q_m$ has a solution $(l,m)$ with $m\ge 3,$ then $p\leq 111$ and $l\le 2.54\cdot10^{20}.$
\item For a $D(4)$-triple $\{a,2a,c_2^\pm\}$ with $a=a_p$ $(p\ge 1),$ if the equation $P_l=Q_m$ has a solution $(l,m)$ in Type $a$  with $m\ge 3,$ then $p\le 116$ and $l\le 2.57\cdot10^{21}.$ If the equation $P_l=Q_m$ has a solution $(l,m)$ in Type $b$  with $m\ge 3,$ then $p\le 28$ and $l\le 2.82\cdot10^{20}.$

\end{enumerate}
\end{lemma}

\textbf{$\bullet$ The case $b=3a.$} Using \eqref{sequence-a_p}, we get
\begin{align}\label{a_p_2}
a=a_p=\dfrac{1}{\sqrt{3}}\left((2+\sqrt{3})^p-(2-\sqrt{3})^p \right).
\end{align}
Combining this with Proposition~\ref{proposition-1}, Lemmas~\ref{lower-Type-a} and \ref{lower-Type-b}, we obtain the following result.

\begin{lemma}
\begin{enumerate}[1)]
\item For a $D(4)$-triple $\{a,3a,c_1^\pm\}$ with $a=a_p$ $(p\ge 1),$ if the equation $P_l=Q_m$ has a solution $(l,m)$ with $m\ge 3,$ then $p\leq 149$ and $l\le 2.56\cdot10^{20}.$

\item For a $D(4)$-triple $\{a,3a,c_2^\pm\}$ with $a=a_p$ $(p\ge 1),$ if the equation $P_l=Q_m$ has a solution $(l,m)$ in Type $a$  with $m\ge 3,$ then $p\le 156$ and $l\le 2.6\cdot10^{21}.$ If the equation $P_l=Q_m$ has a solution $(l,m)$ in Type $b$  with $m\ge 3,$ then $p\le 37$ and $l\le 2.75\cdot10^{20}.$

\end{enumerate}
\end{lemma}

\textbf{$\bullet$ The case $b=6a.$} By \eqref{sequence-a_p}, we have
\begin{align}\label{a_p_3}
a=a_p=\dfrac{1}{\sqrt{6}}\left((5+2\sqrt{6})^p-(5-2\sqrt{6})^p \right).
\end{align}
Combining this with Proposition~\ref{proposition-1}, Lemmas~\ref{lower-Type-a} and \ref{lower-Type-b} we obtain the following result.

\begin{lemma}
\begin{enumerate}[1)]
\item For a $D(4)$-triple $\{a,6a,c_1^\pm\}$ with $a=a_p$ $(p\ge 1),$ if the equation $P_l=Q_m$ has a solution $(l,m)$ with $m\ge 3,$ then $p\le 85$ and $l\le 2.52\cdot10^{20}.$

\item For a $D(4)$-triple $\{a,6a,c_2^\pm\}$ with $a=a_p$ $(p\ge 1),$ if the equation $P_l=Q_m$ has a solution $(l,m)$ in Type $a$  with $m\ge 3,$ then $p\le 89$ and $l\le 2.56\cdot10^{21}.$ If the equation $P_l=Q_m$ has a solution $(l,m)$ in Type $b$  with $m\ge 3,$ then $p\le 22$ and $l\le 3\cdot10^{20}.$

\end{enumerate}
\end{lemma}

For the remaining cases, we will use the following lemma which is a slight modification of the original version of Baker-Davenport reduction method (see \cite[Lemma~5a]{Dujella-Pethoe:1998}). 

\begin{lemma}\label{Dujella-Pethoe}
Assume that $M$ is a positive integer. Let $p/q$ be a convergent of the continued fraction expansion of $\kappa$ such that $q>6M$ and let 
$$ \eta=\parallel\mu q\parallel-M\cdot\parallel\kappa q\parallel,$$ 
where $ \parallel\cdot\parallel $ denotes the distance from the nearest integer. If $\eta>0,$ then there is no solution of the inequality 
$$0<l\kappa-m+\mu<AB^{-l}$$ 
in integers $l$ and $m$ with 
$$\frac{\log(Aq/\eta)}{\log(B)}\leq l\leq M.$$
\end{lemma}

 Dividing $0<\Lambda<11.7\beta^{-2m}$ and $0<\Lambda<4.4a^2\beta^{-2m}$ by $\log\beta$ and using the fact that we have $\beta^{-2m}<\alpha^{-l}$ leads to an inequality of the form \begin{equation}\label{last}
0<l\kappa-m+\mu<AB^{-l},
\end{equation} 
where we consider for solutions of Type $a$
$$
\kappa:=\frac{\log\alpha}{\log\beta},\quad\mu:=\frac{\log\gamma}{\log\beta},\quad A:=\frac{11.7}{\log\beta},\quad B:=\alpha,
$$
and for solutions of Type $b$
$$
\kappa:=\frac{\log\alpha}{\log\beta},\quad\mu:=\frac{\log\gamma}{\log\beta},\quad A:=\frac{4.4a^2}{\log\beta},\quad B:=\alpha.
$$

For the remaining proof, we use Mathematica to apply Lemma~\ref{Dujella-Pethoe}. For the computations, if the first convergent such that $q>6M$ does not satisfy the condition $\varepsilon>0,$ then we use the next convergent until we find the one that satisfies the conditions. The following table provides the information on the results obtained.
\vspace{5mm}
$$
\begin{tabular}{|c|c|c|c|c|c|c|c|}
\hline $c$ &$b=2a$ & $b=3a$ & $b=6a$     \\  \hline \hline
$c_1^{\pm}$ & $l\le 4$ & $l\le 4$ & $l\le 4$     \\ \hline 
$c_2^\pm$ & $l\le 5$ & $l\le 5$ & $l\le 4$      \\ \hline
\end{tabular}
$$
\vspace{5mm}

Therefore, in all cases, we can conclude that 
\begin{align}\label{Upper bound l}
3\le m\le l\le 5.
\end{align}
Combining this with Lemma~\ref{lower-Type-a} and the relations \eqref{a_p_1}, \eqref{a_p_2}, and \eqref{a_p_3}, for solutions in Type $a$, we get 
$$
p\le   \begin{cases}
8,&\ k=2,\\
11,&\ k=3,\\
6,&\ k=6,
\end{cases}
$$
and therefore the equation $Q_m=P_l$ has no solutions in this range. Combining now \eqref{Upper bound l} with Lemma~\ref{lower-Type-b}, in all cases, we get $a\le 98$, which contradicts the fact that $b=ka>10^5$, with $k=2,3,6.$ Finally, we need to see what happens when $m\in \{0, 1,2\}.$ For $m=0,$ we get $x=Q_0=P_0=2$, which gives $d=0.$ By following what is done in the proof of \cite[Lemma 3.3]{aft}, one can easily see that the only solution of equation \eqref{eq-Q_m=P_l} if $m\in \{1, 2\}$ is $(l,m)=(\nu, 1)$ for $z_0=t$ i.e. $x=Q_1=P_{\nu +1}.$ In this case we have 
$$
x=Q_1=P_{\nu +1}=r(T_\nu\pm aU_\nu)+a(bU_\nu \pm T_\nu),
$$
which implies $d=(x^2-4)/a=d_\pm.$

\subsection{Proof of Theorem \ref{main result} with $c=c_3^\pm$}

In the case $c=c_3^\pm$, recall that the problem is more complicated to solve if $s\equiv \pm a\pmod{r}$ by considering the equation \eqref{eq-Q_m=P_l}.  The difficulty lies in the fact that it is not easy to find a lower bound for $l$ and $m$ in term of $a$ in the equation \eqref{eq-Q_m=P_l} if $s\equiv \pm a\pmod{r}.$ To overcome this situation, we will deal with this case by examining the equation $z=v_m=w_n$ using Lemma~\ref{lem:inital_terms}. Now, we will give the lower bounds of the indices $m$ and $n$ in the equation $z=v_m=w_n$, for $2<n<m<2n$ (where the relationship between $m$ and $n$ follows from \cite[Lemma 5]{Filipin-2009} if $m$ and $n$ have the same parity). First, we have the following result.

\begin{lemma}\label{lower-n-m}
\begin{enumerate}[i)]
    \item If the equation $z=v_{2m}=w_{2n}$ has a solution $(m,n)$ with $n > 1,$ then $m>0.495b^{-0.5}c^{0.5}.$
    \item If the equation $z=v_{2m+1}=w_{2n+1}$ has a solution $(m,n)$ with $n > 1,$ then $m^2>0.0625b^{-1}c^{0.5}.$
\end{enumerate}
\end{lemma}

\begin{proof}
$i)$ The statement follows from the proof of \cite[Proposition 2.3]{BaFi}. Here we only have to use that $b>10^5>10^4.$

$ii)$ Using Lemma~\ref{lem:inital_terms} in the case of odd indices and from \cite[Lemma~12]{Filipin-2009}, we get what follows
\begin{equation}
\pm\dfrac{1}{2} astm(m+1)+r(2m+1)\equiv \pm\dfrac{1}{2} bstn(n+1)+r(2n+1) \pmod{c}.
\end{equation}
Because $(st)^2\equiv 16\pmod{c},$ we conclude that $st\equiv \pm 4\pmod{c'}$, for some $c'$ which is a divisor of $c$ and $c'\geq \sqrt{c}.$ Also the $\pm$ sign means that one of the congruences is true. Hence, we get 
\begin{equation}\label{eq-26}
\pm2am(m+1)+r(2m+1)\equiv \pm2bn(n+1)+r(2n+1) \pmod{c'}.
\end{equation}
Let us now assume the opposite i.e., $m^2\le 0.0625b^{-1}c^{0.5}.$ Then, it is easy to see that both sides of the congruence relation $(\ref{eq-26})$ are less than $c'$ and they have the same sign. More precisely, we have 
$$
\max \bigl( 2am(m+1), r(2m+1), 2bn(n+1), r(2n+1) \bigr)\leq 2bm(m+1)
$$ 
and 
$$
2bm(m+1)< 4bm^2\le \frac{c'}{4}.
$$ 
Therefore, we get
$$
|\pm2am(m+1)+r(2m+1)|<\dfrac{c'}{2}\quad \mbox{and}\quad |\pm2bn(n+1)+r(2n+1)|<\dfrac{c'}{2}.
$$ 
Note that in the case of the sign ``-", the two quantities $\pm2am(m+1)+r(2m+1)$ and $\pm2bn(n+1)+r(2n+1)$ are negative and in the case of the sign ``+" they are positive. Thus, we actually have the equations 
\begin{align}\label{rm=rn}
\pm2am(m+1)+r(2m+1)= \pm2bn(n+1)+r(2n+1).
\end{align}
instead of a congruence. Notice that in above $b=ka$ with $k=2,3,6.$ Considering now congruence modulo $a,$ we get
$$
r(2m+1)\equiv r(2n+1)\pmod{a}
$$
which implies 
\begin{align}\label{2m=2n}
2m\equiv 2n\pmod{\dfrac{a}{2}}.
\end{align}
Since $c_3^\pm<519a^5,$ we get 
$$
2n<2m\le 2\cdot 0.0625^{0.5}\cdot b^{-0.5}\cdot \left(519a^5\right)^{0.25}<\dfrac{a}{4}.
$$
Therefore, the congruence \eqref{2m=2n} gives an equation of the form $2m=2n.$ So, from \eqref{rm=rn}, we get $m=n=0$, which is a contradiction. This completes the proof.
\end{proof}
Now, we will combine the lower bounds for indices $m$ and $n$ together with the result obtained using Baker's theory of linear forms in logarithms to prove the main Theorem for large values of $p.$ Using the main result in  \cite{Baker-Wustholz:1993} the third author proved in \cite{Filipin-2008} that $z=v_m=w_n,$ for $n>2,$ implies
\begin{align}
\dfrac{m}{\log (m+1)}< 6.543\cdot 10^{15} \log^2 c.
\end{align}
Combining this with Lemma~\ref{lower-n-m}, in the case of even indices, we get 

\begin{equation}\label{mn-p-even}
\frac{2\cdot0.495b^{-0.5}c^{0.5}}{\log(2\cdot0.495b^{-0.5}c^{0.5}+1)}<6.543\cdot10^{15}\log^2c,
\end{equation}
and in the case of odd indices, we get the inequality
\begin{equation}\label{mn-p-odd}
\frac{2\cdot0.0625^{0.5}b^{-0.5}c^{0.25}+1}{\log(2\cdot0.0625^{0.5}b^{-0.5}c^{0.25}+2)}<6.543\cdot10^{15}\log^2c.
\end{equation}
Therefore, using Maple, the solutions obtained for inequalities \eqref{mn-p-even} and \eqref{mn-p-odd} are summarized in the following lemma.

\begin{lemma}\label{bounds}
\begin{enumerate}[1)]
\item For the $D(4)$-triples $\{a,2a,c_3^\pm\}$ with $a=a_p$ $(p\ge 1)$ defined in \eqref{a_p_1}, if $z=v_{2m}=w_{2n}$ has a solution $(m,n),$ then $p\le 14$ and $m\le 2.6\cdot10^{21}$ but if $z=v_{2m+1}=w_{2n+1}$ has a solution $(m,n),$ then $p\le 40$ and $m\le 2.1\cdot10^{22}.$

\item For the $D(4)$-triples $\{a,3a,c_3^\pm\}$ with $a=a_p$ $(p\ge 1)$ defined in \eqref{a_p_2}, if $z=v_{2m}=w_{2n}$ has a solution $(m,n),$ then $p\le 19$ and $m\le 3\cdot10^{21}$ but if $z=v_{2m+1}=w_{2n+1}$ has a solution $(m,n),$ then $p\le 54$ and $m\le 2.2\cdot10^{22}.$

\item For the $D(4)$-triples $\{a,6a,c_3^\pm\}$ with $a=a_p$ $(p\ge 1)$ defined in \eqref{a_p_3}, if $z=v_{2m}=w_{2n}$ has a solution $(m,n),$ then $p\le 11$ and $m\le 2.6\cdot10^{21}$ but if $z=v_{2m+1}=w_{2n+1}$ has a solution $(m,n),$ then $p\le 30$ and $m\le 2.1\cdot10^{22}.$
\end{enumerate}
\end{lemma}
Now, it remains to see what is happening for small values of $p$ by applying Lemma~\ref{last}.
For this, we also need the inequality which follows from $v_m=w_n, n > 2$ (that is \cite[Lemma~9]{Filipin}),
\begin{align*}
0&<m\log\left(\dfrac{s+\sqrt{ac}}{2} \right)-n\log\left(\dfrac{t+\sqrt{bc}}{2} \right)+\log \dfrac{\sqrt{b}(x_0\sqrt{c}+z_0\sqrt{a})}{\sqrt{a}(y_1\sqrt{c}+z_1\sqrt{b})}\\
&<2ac\left(\dfrac{s+\sqrt{ac}}{2}\right)^{-2m}.
\end{align*}
Using Lemma~\ref{bounds}, we apply Lemma~\ref{Dujella-Pethoe} considering $c=c_3^\pm$. Note that in the case of even indices we have $z_0=z_1=\pm 2,\; x_0=y_1=2$ and in the case of odd indices we have $x_0=y_1=r,\; z_0=\pm t,\; z_1=\pm s$ and $z_0z_1>0.$ We have done the reduction using Mathematica. In all cases according to $b=ka$ with $k=2,3,6$, after at most $2$ steps of reduction, we see that $z=w_m=w_n$ implies $n \le m \le 2.$ In these small ranges, it is not diﬃcult to check that all solutions of $z=v_m=w_n$ will give the extension of $D(4)$-triple $\{a, b, c\}$ to a quadruple with $d=d_-$ or $d=d_+$. This completes the proof of Theorem~\ref{main result}.

\section{Final conclusion}\label{section_5}

Let us mention in this section that the choice of the $D(4)$-pairs $\{a, ak\}$ with $k\in \{2,3,6\}$ is not random. In fact, the idea is to investigate irregular $D(4)$-quadruples, which contain the pairs $\{a, ka\}$, for any positive integer $k.$ In view of the recent work done by the second author in \cite{mbt}, it appears that if $b=ka$ we must have $a<b\le 6.85a$ i.e. $k\in \{2,3,4,5,6\}.$ For $k=4,$ it is easy to check from equation \eqref{pair eqn} that there is no $D(4)$-pair of the form $\{a, 4a\}$. On the other hand, for $k=5$, we obtain an equation of the type
$$
r^2-5a^2=4,
$$
whose positive solutions are defined by $(r, a)=(L_{2n}, F_{2n})$, where $F_n$ and $L_n$ denote the $n$-th Fibonacci and Lucas numbers respectively. Here also the extension of the $D(4)$-pair $\{F_{2n}, 5F_{2n}\}$ should be skipped since it has already been studied (see \cite{dujram}, \cite{Filipin}). In short, it remains to study $k=2,3,6$, which is the aim of this paper. 

\section*{Acknowledgements}
The first author is supported by IMSP, Institut de Math\'ematiques et de Sciences Physiques de l'Universit\'e d'Abomey-Calavi. The second and the third authors are supported by the Croatian Science Foundation, grant HRZZ-IP-2018-01-1313. The fourth author is partially supported by Purdue University Northwest.


%
Institut de Math\'ematiques et de Sciences Physiques, Universit\'e d'Abomey-Calavi, B\'enin \\
Email: adedjnorb1988@gmail.com \\[6pt]
University of Split, Faculty of Science, Ru\dj{}era Bo\v{s}kovi\'{c}a 33,
21000 Split, Croatia\\
Email: marbli@pmfst.hr \\[6pt]
Faculty of Civil Engineering, University of Zagreb, Fra Andrije
Ka\v{c}i\'{c}a-Mio\v{s}i\'{c}a 26, 10000 Zagreb, Croatia \\
Email: filipin@grad.hr \\[6pt]
Department of Mathematics and Statistics, Purdue
University Northwest, 1401 S, U.S. 421, Westville IN 46391 USA \\
Email: atogbe@pnw.edu\\[6pt]

\end{document}